\documentclass{amsart}[12pt]
\usepackage{amssymb,amsmath}
\usepackage{enumerate}
\usepackage[english]{babel}
\usepackage{color}

\newtheorem{lemma}{Lemma}
\newtheorem{thm}{Theorem}

\newtheorem{cor}[thm]{Corollary}
\newtheorem{rem}[thm]{Remark}

\title{Lie derivations on the algebras of locally measurable  operators}
\begin{document}
\date{August 09, 2016}

\begin{abstract}
We prove that every Lie derivation on a solid $*$ - subalgebras of
locally measurable operators  it is
equal to a sum of the  associative derivation and the
center-valued trace.
\end{abstract}

\author{VLADIMIR CHILIN}
\address{National University of Uzbekistan, 100174, \ Tashkent, Uzbekistan}
\email{vladimirchil@gmail.com; chilin@ucd.uz}
\author{ILKHOM JURAEV}
\address{Bukhara State University, 100200, \
Bukhara, Uzbekistan}
\email{ ijmo64@mail.ru}

\maketitle

\section{Introduction}

 Let $A$ be an arbitrary associative algebra and $Z(A)$ be
 the center of the algebra $A.$ A linear operator $D: A\to A$ is called
 (an associative)   derivation if it satisfies the identity
$D(xy)=D(x)y+xD(y)$ for all $x,y \in A$ (Leibniz rule). For any
derivation $D: A\to A$ the inclusion $D(Z(A))\subset Z(A)$ holds.
Every element $a\in A$ defines  a derivation $D_a$ on $A$ given
by $D_a(x)=ax-xa=[a,x],$ $x\in A$. This derivation $D_a$ is said
to be the inner derivation.

A linear operator $L: A\to A$ is called the \textit{Lie
derivation}, if $L\left(\left[x,y\right]\right) =\left[
L(x),y\right]+\left[ x,L(y)\right]$ for all $ x,y \in A$. \ It is
obvious that every an associative derivation $D: A\to A$ is a Lie
derivation. The example of non-associative Lie derivations is the
center-valued trace $E:A\to Z(A)$, i.e. a linear map $E:A\to Z(A)$
such that $E(xy)=E(yx)$  for all $x,y\in A.$

It is well known that any Lie derivation of $L$ on a $C^{*}$ -
algebra $A$ can be uniquely represented in the form $L=D+E$, where
$D$ is ( an associative) derivation and $E$ is a center-valued
trace on $A$ \cite{MV}. Such representation of the Lie derivation $L$
is called the standard form of $L$. In  case when $A$ is a von
Neumann algebra the standard form of the Lie derivation $L:
A\rightarrow A$ has the form $L=D_{a}+E$ for some $a\in A$
\cite{M1}.

Development of the theory of algebras of measurable operators
$S(M)$ and of algebras of locally measurable operators  $LS(M)$
affiliated with  von Neumann algebras or $AW^{*}$ algebras $M$ (see for example
\cite{B},  \cite{MC}, \cite{S1}, \cite{S2}, \cite{San}, \cite{Seg}, \cite{Y}) provided an opportunity to construct and to study new meaning examples of $\ast$ -algebras of
unbounded operators.

One of the interesting problem  is to describe all derivations which act in the algebras  \ $S(M)$ \ and \ $LS(M)$. In case when $M$ is a commutative von Neumann algebra
the equality $S(M)=LS(M)$ holds. In \cite{BCS1} it is proved that any
derivation  in $S(M)$ is inner i.e. trivial, if and only if $M$ is
an atomic algebra. For a commutative $AW^{*}$ - algebra $M$ the
criterion for the existence of nonzero derivations in $S(M)$ is
the lack of the $\sigma$-distributive property of Boolean algebra
of all projections of $M$ \cite{K}.

In case of type I von Neumann algebra    all derivations on the
algebras $LS(M)$ and $S(M)$ are described in \cite{AAK}. In case when $M$
is a properly infinite von Neumann algebra it is proved in \cite{BCS2}
that any derivation on $LS(M)$ and $S(M)$ are inner.

Following the approach of \cite{M1}, in this paper we present  the standard
form of the Lie derivation acting on an arbitrarily solid  $*$-subalgebras of $LS(M)$,
which contains $M.$

We use terminology and notations from the von Neumann algebra
theory \ \cite{D}, \ \cite{S}, \ \cite{SZ} \ and the theory of
locally measurable operators \ \cite{MC}, \ \cite{San}, \
\cite{Y}.

\section{Preliminaries}

Let $H$ be a complex Hilbert space over the field $\mathbb{C}$ of
all complex numbers and let $B(H)$   be the algebra of all bounded
linear operators on $H$. Let  $M$ be a von Neumann subalgebra  in
$B(H)$ \ and \ let \ $\mathcal{P}(M)$ \ be  the
lattice of all projections in $M,$ \ i. e. $\mathcal{P}(M)=\{p\in M:
p^{2}=p=p^{*}\}$. \ Denote by $\mathcal{P}_{fin}(M)$  the
 sublattice of all  finite projections of $\mathcal{ P}(M).$ \ Let $Z(M)$ be
 the center of algebra $M$ and  $\mathbf{1}$  be the identity
in $M$.

A linear subspace $\mathcal{D}$ in $H$  is said to be
\textit{affiliated} with $M$ (denoted as $\mathcal{D} \ \eta \
M$), if $u(\mathcal{D}) \subseteq \mathcal{D}$ for every unitary operator \
$u$ \ from the commutant
$$
M'=\left\{ y\in B(H): xy=yx, \ \ \ \text{for all} \ \ \ x\in M \right\}
$$
of the von Neumann algebra $M.$

A linear subspace $\mathcal{D}$ in  $H$ is said to be strongly
dense in $H$ with respect to the von Neumann algebra $M,$ if
$\mathcal{D} \ \eta \ M$ and there exists a sequence of
projections $\left\{ p_n \right\}_{n=1}^\infty\subset\mathcal{
P}(M),$ such that $p_n \uparrow \mathbf1,$ $p_n(H) \subset
\mathcal{D}$ \ and  \ $p_n^\perp:=\mathbf1-p_n \in \mathcal{P}_{fin}(M)$
\ for all $n\in \mathbb N,$ where $\mathbb N$
denotes the set of all natural numbers.

A linear operator $x$ on $H$ with a dense domain $\mathcal{D}(x)$
is said to be \textit{affiliated} with $M$ (denoted as $x \ \eta \
M$) if $\mathcal{D}(x) \ \eta \ M$ and $ux(\xi)=xu(\xi)$ for all \
$\xi\in \mathcal{D}(x)$ \ and for every unitary operator $u\in
M^{'}$.

A closed linear operator $x$ acting in the Hilbert space $H$ is
said to be measurable with respect to the von Neumann algebra $M$,
if $x \ \eta \ M$ and $\mathcal{D}(x)$ is strongly dense in $H$.
By $S(M)$ we denote the set of all measurable operators with
respect to $M$.

The set $S(M)$ is a  a unital $\ast$-algebra  with respect algebraic operations of strong addition and multiplication and taking the adjoint of an operator (it is assumed that the multiplication by a scalar defined as usual wherein  $0\cdot x=0$) \
\cite{Seg}.

A closed linear operator $x$ in $H$ is said to be \textit{ locally
measurable with respect to the von Neumann algebra } $M$, if $x \
\eta \ M$ and there exists a sequence  $\left\{ z_n
\right\}_{n=1}^\infty$ \  of central projections in $M$ such that
$z_n \uparrow \mathbf 1$ \ and \ $xz_n\in S(M)$  for all
$n\in\mathbb N$.

The set $LS(M)$ of all locally measurable operators with respect
to $M$ also is a unital $\ast$-algebra equipped with the algebraic
operations of strong addition and multiplication and taking the
adjoint of an operator. In this case, $S(M)$ and $M$ is a $\ast$-subalgebras in $LS(M)$ \cite[Ch. II, \S 2.3]{MC}. Center $Z(LS(M))$ in
the $\ast$-algebra $LS(M)$ coincides with the $\ast$-algebra $S(Z(M)).$ In
the case when $M$ is a factor or $M$ is finite von Neumann algebra, the equality $LS(M)=S(M)$ holds.

Let $x$ be a closed operator with the dense domain $\mathfrak{D}(x)$
in $H$, let $x=u|x|$ be the polar decomposition of the operator
$x$, where $|x|=(x^*x)^{\frac{1}{2}}$ and $u$ is a  partial
isometry in $B(H)$ such that $u^*u$ is the right support $r(x)$ of
$x$. It is known that $x\in LS(\mathcal{M})$ (respectively, $x\in
S(\mathcal{M})$) if and only if $|x|\in LS(\mathcal{M})$
(respectively, $|x|\in S(\mathcal{M})$) and $u\in
\mathcal{M}$~\cite[Ch. II, \S\S\,2.2, 2.3]{MC}.

A $\ast$-subalgebra $A$ in $LS(M)$ is called the solid $\ast$-subalgebra of  $LS(M)$,
 if $MAM \subset A$, i.e. $axb\in A$ for any $a,b\in M$ and $x\in  A$. It is known  that a $\ast$-subalgebra \ $A$ \ of \ $LS(M)$ \
is  solid if and only if the  condition $x\in LS(M)$, $y\in A,$ $|x|\leq
|y|$ implies that $x\in A$ \ (see for example, \cite{BPS}).
The examples of  solid $\ast$-subalgebras of $LS(M)$ are  $\ast$-subalgebras  $M$ and $S(M)$.

\section{Standard decomposition of Lie derivations}

In this section the standard decomposition of a Lie derivation
acting in a solid  $*$-subalgebra $A$ in $LS(M)$ containing $M$ is
established. More precisely, we have the following

\begin{thm}
\label{thm1}
Let $A$ be a solid  $*$- subalgebra in $LS (M)$ containing $M,$
and $L$ be a Lie derivation on $A.$ Then there exist (associative)
derivation $D:A\rightarrow A$ and center-valued trace  $E: A\to
Z(A)$ such that
$$L(x)=D(x)+E(x) \eqno (3.1)$$ for all $x\in A$.
\end{thm}

\begin{rem}\label{r1} In the case of algebras $LS(M)$  and $S(M)$  the standard decomposition
 (3.1) for the Lie derivation $L$ in general is not unique.    For example, if $M$ is a commutative von Neumann algebra without atoms, there are uncountably many distinct associative
derivations  $D:S(M)\rightarrow S(M)$ \cite{BCS1}, \cite{Ber},  which are at the
same time $S(M)$ - valued traces. Therefore, taking two different
nonzero derivations $D$ and $D_{1}$ on $S (M),$ we have that for a
zero Lie derivation the equality
$$L=D+(-D)=D_{1}+(-D_{1})
$$
holds,  i. e. the standard decomposition for $L$
is not unique.
\end{rem}

\begin{rem}\label{r2} If $M$ is  a  commutative von Neumann algebra, then
the $\ast$-algebra $LS(M)$ is also commutative \cite[Ch. II, \S 2.2]{MC},
and therefore for any $*$ - subalgebra $A$ in $LS (M),$ we have
that $Z (A) = A.$ Hence, in this case, the class of Lie
derivations on $A$ coincides with the class \ $Z(A)$-valued trace
on $A$.
\end{rem}

For the prove Theorem \ref{thm1}, we first consider the case when the von
Neumann algebra $M$ has no direct commutative summands, i.e. for any
nonzero central projection $z\in\mathcal{P}(Z(M))$ the von Neumann
algebra $zM$ is not commutative. In this case, in the von Neumann
algebra $M$, there exists a non-zero projection $p$ such that
$$
c(p)=c(\mathbf{1}-p) = \mathbf{1} \ \ \text{and} \ \ \ p \preceq \mathbf{1} - p,
$$
where \
$c(p):=\mathbf{1}-\sup\{z\in \mathcal{P}(Z(M)):pz=0\}$ \ is the
central support of projection \ $p$ \ ~\cite[Task 6.1.9]{KR}.

Let $A$ be an   arbitrary solid $\ast$-subalgebra in $LS(M)$ such
that \ $M\subset A$. Let $p_{1}=p$, $p_{2}=\mathbf{1}-p.$ Consider
 subalgebras $S_{ij}=p_iAp_j= \left\{ p_ixp_j: x\in A\right\}$ \
in \ $A, \ i,j=1,2$. It is clear that $S_{ik}S_{lj}\subset S_{ij}$
\ and the inclusion \ $M\subset A$ implies that
$p_{i}Mp_{j}\subset S_{ij}$ for any $i,j,k,l=1,2$. In this case,
taking into account the  equality
$$
x=p_{1}xp_{1}+p_{1}xp_{2}+p_{2}xp_{1}+p_{2}xp_{2}, \ \  x \in A,
$$
we obtain that

$$
A=\sum\limits_{i,j=1,2} p_i Ap_j.
$$

Moreover, for \  $x\in S_{ik}$, \ $y\in S_{lj}$ \ the inclusion \
$xy \in S _{ij}$\ holds, and if $k\neq l$, then $xy=0$, \  i.e. in
this case, \ $S_{ik} S_{lj}={0}$.

For the prove Theorem \ref{thm1} we need some technical lemmas,
which are similar to the corresponding lemma in ~\cite{Mar},
~\cite{M1}. The proof of these lemmas is similar to those in
~\cite{Mar}, ~\cite{M1} and are given here for the sake of
completeness  (in contrast with ~\cite{Mar}, ~\cite{M1} we also consider the
case when $A\neq M$).

\begin{lemma}
\label{lemma1} (see ~\cite[Lemma 1]{M1}).  If $x\in S_{ij}$ and
$xy=0$ for all $y\in S_{jk}$, \ then $x=0$.
\end{lemma}
\begin{proof} If $j=k$, then for \ $p_{j}\in S_{jj}$
\ we have, that \ $x=xp_{j}= 0$. \
 Let $i=j=1,k=2$, \ $x\in S_{11}$ \ and  \ $xy=0$ \ for all \ $y\in S_{12}$.
  \ Since $p_1 = p \preceq \mathbf{1} - p = p_2$   \ there exists  projection \
 $q_{1}\leq p_{2}$ \ such that  \ $p_{1}\sim q_{1}$, \
i.e. \ $u^{*}u=p_{1}, \ uu^{*}=q_{1}$ \ for some partial isometry
 \ $u$ \ from \ $M$. Taking into account \
$u^{*}q_{1}u=p_{1}, \ q_{1}=q_{1}p_{2}, \ x=p_{1}xp_{1}$ \ and
inclusion\ $M\subset A$, \ we have, that \
$p_{1}u^{*}q_{1}p_{2}\in S_{12}$ �
$$
x=p_{1}xp_{1}=p_{1}xp_{1}u^{*}u=x(p_{1}u^{*}q_{1}p_{2})u=0.
$$ For other indices $i,j$ the proof of the lemma is similar.

\end{proof}

\begin{lemma}
\label{lemma2} Let \ $p\in \mathcal{P}(M)$ \ and \ let \ $L:
A\rightarrow A$ \ be a Lie derivation then
$$
x( pL\left( p\right)+L\left( p\right) p+pL\left( p\right)p
-L\left( p\right))-( pL\left( p\right)+L\left( p\right)p+pL\left(
p\right)p-L\left( p\right))x=
$$
$$
= 3px( pL\left( p\right)+L\left(
p\right)p-L\left( p\right))-3( pL\left( p\right)+L\left(
p\right)p-L\left( p \right)) xp \eqno (3.2)
$$
for every $x\in A$.
\end{lemma}

The proof is exactly the same as the proof of Lemma 4 in \
\cite{Mar}.

\begin{lemma}
\label{lemma3} (see \ \cite[Lemma 5]{Mar},
 \ \cite[Lemma 4]{M1}).
 Let \ $p\in \mathcal{P}(M)$, \ $c(p)=c(\mathbf{1}-p)$ \ and let \ $L:A\rightarrow A$ \ be a
Lie derivation. Then \ $L(p)=[p,a]+z$ \ for some \  $a\in A$ \ and \ $z\in Z(A).$
\end{lemma}
\begin{proof} Let \ $p_{1}=p, \ p_{2}=\mathbf{1}-p$, \
$S_{ij}=p_{i}Ap_{j}$. \ Since \ $A=\sum\limits_{i,j=1,2}S_{ij}$, \
then \ $L(p)=\sum\limits_{i,j=1,2} x_{ij}$, \ where \  $x_{ij}\in
S_{ij}$. \
 Applying (3.2) for \ $x\in A$ \ we obtain that
$$x(p(x_{11}+x_{12}+x_{21}+x_{22})+(x_{11}+x_{12}+x_{21}+x_{22})p+$$
$$+p(x_{11}+x_{12}+x_{21}+x_{22})p-(x_{11}+x_{12}+x_{21}+x_{22}))-$$
$$(p(x_{11}+x_{12}+x_{21}+x_{22})+(x_{11}+x_{12}+x_{21}+x_{22})p+$$
$$+p(x_{11}+x_{12}+x_{21}+x_{22})p-(x_{11}+x_{12}+x_{21}+x_{22}))x=   \eqno (3.3)$$
$$=3px(p(x_{11}+x_{12}+x_{21}+x_{22})+(x_{11}+x_{12}+x_{21}+x_{22})p-$$
$$(x_{11}+x_{12}+x_{21}+x_{22}))-3(p(x_{11}+x_{12}+x_{21}+x_{22})+$$
$$(x_{11}+x_{12}+x_{21}+x_{22})p-(x_{11}+x_{12}+x_{21}+x_{22}))xp.$$
Since \ $x_{ij}\in  S_{ij}$, \ $i,j=1,2$, \ it follows that\
$x_{ij}=p_{i}x_{ij}p_{j}$. \ In particular,
$$
px_{1j}=x_{1j}, \ x_{i1}p=x_{i1}, \  px_{11}p= x_{11}, \ px_{ij}p=0
$$
��� $i\neq1$, ���� $j\neq1$. \ Thus, equality (3.3) implies that
$$x(x_{11}+x_{12}+x_{11}+x_{21}+x_{11}-x_{11}-x_{12}-x_{21}-x_{22})-$$
$$-(x_{11}+x_{12}+x_{11}+x_{21}+x_{11}-x_{11}-x_{12}-x_{21}-x_{22})x=$$
$$=3px(x_{11}+x_{12}+x_{11}+x_{21}-x_{11}-x_{12}-x_{21}-x_{22})-$$
$$-3(x_{11}+x_{12}+x_{11}+x_{21}-x_{11}-x_{12}-x_{21}-x_{22})xp.$$
Thus
$$
x(2x_{11}-x_{22})-(2x_{11}-x_{22})x=3px(x_{11}-x_{22})-
3(x_{11}-x_{22})xp. \eqno (3.4)
$$
If \ $x\in S_{12}$, \ then \ $x_{22}x=xx_{11}=0$ \ and \ $xp=0$, \
equality (3.4) implies that \ $x_{11}x=xx_{22}$. \ Hence,
$$
(x_{11}+x_{22})x=x(x_{11}+x_{22}).
$$
Similarly, for\ $x\in S_{21}$ \ we have that
$$
(x_{11}+x_{22})x=x(x_{11}+x_{22}).
$$
 Now let \ $x\in S_{11}$ \ and \ $y\in
S_{12}$. Since \ $xy\in S_{12}$ it follows that
$$
(x_{11}x-xx_{11})y = (x_{11}+x_{22})xy-x(x_{11}+x_{22})y=
$$
$$=(x_{11}+x_{22})xy-xy(x_{11}+x_{22})=
(x_{11}+x_{22})xy-(x_{11}+x_{22})xy=0,
$$
i. e. \ $ (x_{11}x-xx_{11}) y =0$ \ for all \ $x\in S_{11}$ \ and
\ $y\in
 S_{12}$. \ Since $x_{11}x-xx_{11}\in S_{11}$, it follows from Lemma \ref{lemma1} that
$$
(x_{11}+x_{22})x-x(x_{11}+x_{22})=0
$$
for all $x\in S_{11}$.

Similarly, in the case of \ $x\in S_{22}$,
\ we have that
$$
(x_{11}+x_{22})x-x(x_{11}+x_{22})=0.
$$
Thus, \
$$(x_{11}+x_{22})b=b(x_{11}+x_{22})
$$
for all \ $b\in S_{ij}$, \ where \ $i,j=1,2$, \ i.e. \
$x_{11}+x_{22} = z$ \ for some \ $z\in Z(A)$. \ Hence, \
$L(p)=(x_{12}+x_{21})+z$. \ Setting  \ $a=x_{12}-x_{21}$ \
we obtain that   \ $pa-ap=x_{12}+x_{21}$, \ and therefore \
$L(p)=[p,a]+z$.
\end{proof}

\begin{lemma}
\label{lemma4} (see \ \cite[Lemma 6]{Mar}). If \ $p$ \ and \ $L$ \
are the same as in Lemma \ref{lemma3} \ and\ $L(p)\in Z(A)$, \
then  \ $L\left(  S_{ij}\right) \subset S_{ij}$ \ for \ $i\ne j$.
\end{lemma}
\begin{proof} Let  \ $x \in  S_{12}$ \ and \ $L(x)=\sum\limits_{i,j=1,2} x_{ij}$,
 \ where \  $x_{ij}\in S_{ij}$. \ Since \ $[p,x]=x$, \ it follows from the inclusion \ $L(p)\in
Z(A)$, \ that
$$
\sum\limits_{i,j=1,2}x_{ij}=L(x)=L([p,x])=[L(p),x]+[p,L(x)]=[p,L(x)]=x_{12}-
x_{21}.
$$
Therefore, \ $x_{11}+2x_{21}+x_{22}=0$, \ and hence
$x_{11}=p(x_{11}+2x_{21}+x_{22})=0$. Thus \ $2x_{21}+x_{22}=0$, \
which implies that \
$$
x_{22}=x_{22}(\mathbf{1}-p)=(2x_{21}+x_{22})(\mathbf{1}-p)=0.
$$
Therefore, \ $x_{22}=0$, \ \ and so \ $x_{21}=0$. \ Thus \ $L(x)
\in S_{12}$. \  Similarly, the inclusion \ $L(S_{21})\subset
S_{21}$ is established.
\end{proof}

\begin{lemma}
\label{lemma5} (see \ \cite[Lemma 7]{Mar}). If \ $p$ \ and \ $L$ \
are the same as in Lemma \ref{lemma3}, \ and \ $L(p)\in Z(A)$, \
then \ $L\left(  S_{ii}\right) \subset S_{ii}+Z(A), \ i=1,2$.
\end{lemma}

\begin{proof}  If \ $x\in  S_{11}$ \ and \
$L(x)=\sum\limits_{i,j=1,2} x_{ij}$, \ where \ $x_{ij}\in S_{ij}$,
then
$$
x=pxp, \ [p,x]=px-xp=0,
$$
and therefore
$$0=L([p,x])=
[L(p),x]+[p,L(x)]=[p,L(x)]=x_{12} -x_{21}.
$$
Consequently,
$$
x_{12}=px_{12}=p(x_{12}-x_{21})=0
$$
and
$$x_{21}=x_{12}-(x_{12}-x_{21})=0,$$
i.e. \ $L(x)\in S_{11}+ S_{22}$. \ Similarly, if \ $y\in S_{22}$,
\ then \ $L(y)\in  S_{11}+
 S_{22}$.

Let now \ $x\in  S_{11}$ \ and \ $y\in  S_{22}$, \
$L(x)=a_{11}+a_{22}, \ L(y)=b_{11}+b_{22}$, \ where \ $a_{ii}, b_{ii} \in
S_{ii}, \ i=1,2$. \ From equations \
$0=L([x,y])=[a_{22},y]+[x,b_{11}]$ \ and  \ $[a_{22},y]\in S_{22},
\ [x,b_{11}]\in S_{11},$ \ it follows that \
$[a_{22},y]=0=[x,b_{11}]$. In particular, \ $a_{22}b=ba_{22}$ \
for all \ $b\in (\mathbf{1}-p)A(\mathbf{1}-p)$,  i. e. \
$a_{22}\in Z((\mathbf{1}-p)A(\mathbf{1}-p))$.

Let us show that\
$Z((\mathbf{1}-p)A(\mathbf{1}-p))=(\mathbf{1}-p)Z(A)$. \ It is
clear that
$$
(\mathbf{1}-p)Z(A)\subset Z((\mathbf{1}-p)A(\mathbf{1}-p)).
$$
Let\ $0\leq a\in Z((\mathbf{1}-p)A(\mathbf{1}-p)), \
z=a+(\mathbf{1}-p)$ \ and let \ $b$ \ be an inverse of \ $z$ \ in the
algebra \ $(\mathbf{1}-p)LS(M)(\mathbf{1}-p)$. \  Since \ $0\leq
b\leq\mathbf{1}-p$ \ and \ $b\in
Z((\mathbf{1}-p)A(\mathbf{1}-p))$, it follows that \ $b\in
Z((\mathbf{1}-p)M(\mathbf{1}-p))$. \ According to \ \cite[ Part I,
Ch. 2, \S 1]{D}, \ the equality \
$Z((\mathbf{1}-p)M(\mathbf{1}-p))=(\mathbf{1}-p)Z(M)$\  holds, and
therefore \ $b=(\mathbf{1}-p)d$, \ where \ $d\in Z(M)\subset
Z(A)$. \ Consequently,\ $b\in(\mathbf{1}-p)Z(A)$ \ and the \
operator \ $z$ \ is inverse of \ $b$ \  in the algebra \
$(\mathbf{1}-p)LS(M)(\mathbf{1}-p)$ \ also belongs of the algebra
\ $(\mathbf{1}-p)Z(A)$, \ that implies $a\in (\mathbf{1}-p)Z(A)$.
Thus, $Z((\mathbf{1}-p)A(\mathbf{1}-p))=(\mathbf{1}-p)Z(A)$. \
Hence \ $a_{22}=\left( \mathbf1- p \right) z$, \  for some  \
$z\in Z(A)$, \ i.e.
$$
L(x)=a_{11}+\left( \mathbf1- p \right) z=(a_{11}-pz)+z\in
S_{11}+Z(A).
$$

The inclusion  \ $L(S_{22})\subset
 S_{22}+Z(A)$ is established similarly.
\end{proof}

Lemmas \ref{lemma4} and \ref{lemma5} imply that in the case, when
\ $L(p)\in Z(A)$, \ we have that \ $L(x)\in S_{ij}$ \ for \ $x\in
S_{ij},$ \  \ $i\neq j$ \  and that \ $L(x)=d(x)+z(x)$ \ , \ $x\in
S_{ii}$, \ where \ $d(x)\in S_{ii}$ \ and \
 $z(x)\in Z(A)$ .

We chose the projection  \ $p\in\mathcal{ P}(M)$ \ so that the
equality \ $c(p)=c(\mathbf{1}-p)=\mathbf{1}$ \  holds (a
projection exists as a von Neumann algebra \ $M$ \ has no direct
Abelian summand). In this case, the equality  \ $pAp\cap
Z(A)=\{0\}$ holds. \ Indeed, if there is \ $0\neq a\in pAp\cap
Z(A)$, \ then \ $0\neq |a|\in pAp\cap Z(A)$ \ and therefore\ the
operator $|a|$ \ has non-zero spectral projection\ $q=\{|a|\geq
\epsilon\}$, \
  such that \ $q\leq p$ \ and \ $q\in Z(A) $. \ Hence,
 \  $q(\mathbf{1}-p)=0$,  \ which is impossible by equation   \ $c(\mathbf{1}-p)=\mathbf{1}$. \  Thus, \ $pAp\cap Z(A)=\{0\}$.

If \ $x\in S_{ii}$ \ and \ $L(x)=d_{1}(x)+z_{1}(x)$, \ where \
$d_{1}(x)\in S_{ii}, \ z_{1}(x)\in Z(A)$, \ then
$$d(x)-d_{1}(x)=z(x)-z_{1}(x)\in
S_{ii}\cap Z(A)=(p_{i}Ap_{i})\bigcap Z(A)=\{0\}.
$$
Consequently, for each  \ $x\in S_{ii}$ \ the element \ $L(x)$ \
is uniquely represented in the form \ $L(x)=d(x)+z(x)$, \ where \
$d(x)\in S_{ii}, \ z(x)\in Z(A)$. \

Thus, when choosing a projector  \ $p\in \mathcal{P}(M)$ \ with \
$c(p)=c(\mathbf{1}-p) = \mathbf{1}$, \  where \ $L(p)\in Z(A)$, \
the map \ $D$ \ is correctly defined from

$A=\sum\limits_{i,j=1,2}S_{ij}$ \ in \ $A$ \ by the rule
$$
D(x)=d(x_{11})+d(x_{22})+L(x_{12}+x_{21}), \eqno (3.5)
$$
where \ $x_{ij}=p_{i}xp_{j}, \ i,j=1,2,  \ L(x_{12}+x_{21})\in
S_{12}+S_{21}$.

If \ $x,y\in S_{11}$, \ then
$$L(x+y)=d(x+y)+z(x+y);$$
$$L(x)=d(x)+z(x);$$
$$L(y)=d(y)+z(y).$$
Since \ $L$ is a linear map, it follows that
$$d(x+y)+z(x+y)=d(x)+d(y)+z(x)+z(y);$$
$$d(\lambda x)+z(\lambda
x)=\lambda(d(x)+z(x)),\lambda\in \mathbb{C}.$$ Since \
$p_{1}Ap_{1}\cap Z(A)=\{0\}$, \ it follows that \
$$d(x+y)=d(x)+d(y);$$
$$z(x+y)=z(x)+z(y);$$
$$d(\lambda x)=\lambda d(x);$$
$$z(\lambda x)=\lambda z(x).$$
This means that the mapping \ $d:S_{11}\rightarrow S_{11} $ \ and
\  $z: S_{11}\rightarrow Z(A)$ \  are linear. For the same
reasons, these mappings are linear on  \ $S_{22}$. \ Thus the map
\ $D$ \ constructed in (3.5)  is a linear operator acting on the
solid $\ast$-subalgebra $A$. In this case, the map \ $E$ \ from \
$A$ \ to \ $Z(A)$, \ defined by the equality
$$
E(x)=L(x)-D(x)=z(x_{11})+z(x_{22}), \ \  x\in A,
$$
is also linear. In additional, for \ $x\in S_{ij}, \ i\neq j$. \ we have that
\ $E(x)=0$, \ i.e. $L(x)=D(x)$.

\begin{lemma}
\label{lemma6} (see \ \cite[Lemma 9]{Mar}). If \ $x\in S _{ij}, \
i\neq j, \ y\in A$, \ then \
$$
D(xyx)=D(x)yx+xD(y)x+xyD(x).
$$
\end{lemma}

\begin{proof}
Since \ $L(y)=D(y)+E(y)$ \ and \ $E(y) \in Z(A)$, \ it follows
that $[x,L(y)]=[x,D(y)].$

According to the equation
$$
[[x,y],x]=[p_{i}xp_{j}y-yp_{i}xp_{j},p_{i}xp_{j}]=2xyx\in
S_{ij},
$$ we have
$$
2D(xyx)=L(2xyx)=L([[x,y],x])= [[L(x),y]+[x,L(y)],x]+ [[x,y],L(x)]=
$$
$$
=[[D(x),y]+[x,D(y)],x]+[[x,y],D(x)]=
$$
$$
=D(x)yx-yD(x)x-xD(x)y+xyD(x)+xD(y)x-D(y)x^{2}-
$$
$$
-x^{2}D(y)+xD(y)x+xyD(x)-yxD(x)-D(x)xy+D(x)yx.
$$
Since \ $x\in S_{ij}, \ i\neq j$, \ it follows that \ $x^{2}=0$ \
and \ $D(x)=L(x)\in S_{ij}$. \ Therefore,  \  $D(x)x=0=xD(x)$ \
and \ $D(xyx)=D(x)yx+xD(y)x+xyD(x)$.
\end{proof}

\begin{lemma}
\label{lemma7} (see \ \cite[Lemma 10, 11]{Mar}). If $x \in S_{ii}$
and $y\in A$, or $x \in A$ and $y\in  S_{ii},$ $i=1,2$, then
$$D(xy)=D(x)y+xD(y). \eqno (3.6)$$
\end{lemma}

\begin{proof}
If\ $x\in S_{11}$ \ and \ $y\in S_{12}$, \ then \ $xy\in S_{12}, \
[x,y]=xy$, \ and so
$$D(xy)=L(xy)=L([x,y])=[D(x)+E(x),y]+[x,D(y)+E(y)]=$$
$$=[D(x),y]+[x,D(y)]=D(x)y+xD(y).$$
For\ $x\in S_{11}$ \ and \ $y\in S_{21}$ \ we have, that  \
$xy=0$. \ Since \ $D(x)\in S_{11}$ \ and \ $D(y)\in S_{21}$, \ it follows
that $ D(x)y=0, \ xD(y)=0$, \ which implies that (3.6). For cases
\ $x\in S_{22}, \ y\in S_{ij}, \ i\neq j$ \ the equality  (3.6) is
established similarly.

Now let \ $x,y\in S_{11}, \ z\in S_{12}$. \
Since \ $yz\in S_{12}$, \  according to the arguments above it
follows that
$$
D(xy)z=D(xyz)-xyD(z)=D(x)yz+xD(yz)-xyD(z)= D(x)yz+
$$
$$
+x( D(y)z+yD(z)) -xyD(z)= (D(x)y+xD(y)) z.
$$
Consequently,  \ $( D(xy)-D(x)y-xD(y)) z=0$ \ for all\ $z\in
S_{12}$. \ Since  $D(x), \ D(y), \ D(xy) \in S_{11}$, \ it follows
that \ $(D(xy)-D(x)y-xD(y))\in S_{11}$, \    and from Lemma
\ref{lemma1} we have

$$
D(xy)-D(x)y-xD(y)=0.
$$
In the case when, \  $x\in S_{11}, \ y\in S_{22}$, \ or  \
  $x\in S_{22}, \ y\in  S_{11}$, \ we have that\ $xy=0$ \ and \ $D(x)y+xD(y)=0$.
  \ The case \ $x,y \in  S_{22}$ \ can be treated similarly to that of \ $x,y\in S_{11}$. \

Thus, equality (3.6) holds for any  \ $x\in S_{ii}, \ y\in S_{kl}$
\ for all \ $i, k, l=1,2$.  \. Using the equality  \
$A=\sum\limits_{i,j=1,2} S_{ij}$ \ and \  linearity of the mapping
\ $D$, \ we obtain that the equality (3.6) holds for any  \ $x\in
S_{ii}, \ y\in A, \ i=1,2 $. The equality (3.6) for the case \
$x\in A$ \ and \
 $y\in S_{ii}, \ i=1,2$ can be established similarly.
\end{proof}

\begin{lemma}
\label{lemma8} (see \  \cite[Lemma 10]{M2}).
The linear mapping \ $D$, \  defined by (3.5) is an associative
derivation on $A$.
\end{lemma}
\begin{proof}
In the case when  \ $x\in S_{ii}, \ i=1,2, \ y\in A$, \ lemma
\ref{lemma8} is proved in Lemma  \ref{lemma7}. If \ $ x\in S_{12}$
\ and \ $y\in S_{21}$, \ then
$$
E([x,y])=L([x,y])- D([x,y])=[L(x),y]+[x,L(y)]-D([x,y])=
$$
$$
=[D(x),y]+[x,D(y)]-D(xy)+D(yx).
$$
Consequently, for \ $z=E([x,y])\in Z(A)$, \ we have
$$ z={D(x)y+xD(y)-D(xy)}+{D(yx)-D(y)x-yD(x)}. \eqno (3.7)$$
Since
$$
(D(x)y+xD(y)-D(xy))\in S_{11}, \ (D(y)x+yD(x)-D(yx))\in S_{22}
$$
and \ $S_{11}\cap S_{22}=\{0\}$, \ then in the case of  \ $z=0$ \
 equation (3.7) implies that
$$D(xy)=D(x)y+xD(y).$$
Let us show, that \ $z=0$. \ Suppose that \ $z\neq 0.$\
Multiplying the left hand-side of (3.7) by $x$ and then by $y$, we
have that
$$
 xz=xD(yx)-xD(y)x-xyD(x);  \eqno (3.8)
$$
$$
 yz=yD(x)y+yxD(y)-yD(xy).   \eqno (3.9)
$$
Since \ $yx\in S_{22}$, \ it follows from Lemma \ref{lemma7} that

$$D(xyx)=D(x)yx+xD(yx).$$
From this fact and the equality (3.8), it follows that
$$xz=D(xyx)-D(x)yx-xD(y)x-xyD(x).$$
Now, using Lemma  \ref{lemma6}, we have that  \ $xz=0$.\

Due to \
$xy\in S_{11}$ \ and Lemma \ref{lemma7}, we have
$$D(yxy)=D(y)xy+yD(xy).$$

Therefore, from (3.9) and Lemma 6 it follows that
$$yz=yD(x)y+yxD(y)-D(yxy)+D(y)xy=$$
$$yD(x)y+yxD(y)-D(y)xy-yD(x)y-yxD(y)+D(y)xy=0.$$
Let \ $x=u|x|$ be the polar decomposition of locally measurable
operators \ $x\in LS(M)$. \ Since \ $xz=0, \ x\in A, \ z\in Z(A),
\ M \subset A$, \ it follows that
$$
|x|=u^{*}x \in A, \
z|x|=|x|z=u^{*}xz=0,
$$
Hence
$$|x|z^{*}=z^{*}|x|=0 \ \ \ \text{and} \ \ \ xz^{*}=z^{*}x =0.$$
Similarly, using the equality \ $yz=0$, \, we have that
$$
|y|z^{*}=z^{*}|y|=0 \ \ \ \text{and} \ \ \ yz^{*}=z^{*}y =0.
$$
Multiplying the left hand-side and the right hand-side of (3.7) by
\ $z^{*}$\, we have that
$$
z^{*}zz^{*}=z^{*}(D(yx)-D(xy))z^{*}.
$$
Since \
$$yx\in S_{22}, \ uxz^{*}=0=z^{*}y$$
then by Lemma \ref{lemma7}, we have that
$$
0=D(yx(\mathbf{1}-p)z^{*})=D(yx)(\mathbf{1}-p)z^{*}+yxD((\mathbf{1}-p)z^{*})\eqno (3.10).
$$
Using \ $D(yx)\in S_{22}$ \ and equality  (3.10), we have that
$$D(yx)z^{*}=D(yx)(\mathbf{1}-p)z^{*}=-yxD((\mathbf{1}-p)z^{*}).$$
Similarly,  \ $D(xy)z^{*}=-yxD(pz^{*})$. \ Since \
$z^{*}x=0=z^{*}y$, \ it follows that
$$
z^{*}zz^{*}=z^{*}xyD(pz^{*})-z^{*}yxD((\mathbf{1}-p)z^{*})=0,
$$
that implies  \ $|z^{*}|^{2}=zz^{*}zz^{*}=0$, \ i. e. \ $z=0$. In
the case  $x \in S_{21}$ and  $y \in S_{12}$ the statement
 of Lemma \ref{lemma8} is proved similarly.
\end{proof}
In particular, Lemma \ref{lemma8} implies that
\begin{cor}
\label{cor1} $E([x,y])=0$ \ for all \ $x,y\in A$.
\end{cor}
\begin{proof}
Since the  \ $D$ is an associative derivation (Lemma
\ref{lemma8}), it follows that \ $D$ \ is also a Lie derivation,
moreover,
 \ $E(x)=L(x)-D(x)\in Z(A)$ \ for all \ $x\in A$. \
Therefore, for any  \ $x,y\in A$ \ we have
$$L([x,y])=[D(x)+E(x),y]+[x,D(y)+E(y)]=[D(x),y]+[x,D(y)]=D([x,y]),$$
which implies that  \ $E([x,y])=0$.

\end{proof}

Now we are ready to establish the standard  decomposition for the
Lie derivation $L$ in the case, when  von Neumann   algebra \ $M$ \ has
no direct commutative summands.

\begin{thm}
\label{thm2} Let\ $M$ be a von Neumann algebra without direct
commutative summands and let \ $A$ be a solid $\ast$-subalgebra in $LS(M)$, containing $M$. Then any Lie derivation $L$ on $A$ is  of the form \ $L = D + E,$ \ where $D$ is an associative
derivation on \ $A$ \ and \ $E$ \ is a trace on \ $A$ \ with values in \ the center
$Z(A)$.
\end{thm}

\begin{proof}
 Since von Neumann algebra \ $M$ \ has no direct
commutative summands, there exists a projection  \ $p\in
\mathcal{P}(M)$\ such that $c(p)=c(\mathbf{1}-p) =\mathbf{1}$ \
(~\cite[Task 6.1.9]{KR}). Assume first that \ $L(p)\in Z(A)$. \
In this case, according to Lemma 8 and Corollary \ref{cor1},\  we
have that \ $L=D+E$, \, where \ $D$ is an associative derivation
on \ $A$ \ and \ $E$ \ is a trace on \ $A$ \  with values in \
$Z(A)$.

Now let \ $L(p)\bar{\in} Z(A)$. \ By  Lemma \ref{lemma3}, the
equality \ $L(p)=[p,a]+z$ \ holds for some \ $a\in A$ \ and \
$z\in Z(A)$. Consider on \ $A$ \ the  inner derivation \
$D_{a}(x)=[a,x], \ x\in A$,  and let\  $L_{1}=L+D_{a}$. \ It is
clear that  \ $L_{1}$ \ is a Lie derivation  on \
$A$, \ moreover
$$
L_{1}(p)=L(p)+D_{a}(p)=[p,a]+[a,p]+z=z\in Z(A).
$$
By Lemma \ref{lemma8} \ and  Corollary  \ref{cor1} \ the Lie
derivation \ $L_{1}$ \ has the standard form, i.e. \
$L_{1}=D_{1}+E_{1}$, \ where \ $D_{1}$ is an associative
derivation on \ $A$ \ and \ $E_{1}$ \ is a $Z(A)$-valued trace on \
$A$. \ Therefore, \
$$L=L_{1}-D_{a}=D_{1}-D_{a}+E_{1}.$$
It is clear, that  \ $D=D_{1}-D_{a}$ \ is an associative
derivation on \ $A$, \ and consequently
 the equality \ $L=D+E_{1}$ \
completes the proof of Theorem \ref{thm2}.
\end{proof}
We now consider the case of an arbitrary von Neumann algebra  \
$M$.  We need the following

\begin{lemma}
\label{lemma9}
 Let \ $M$ \ be a arbitrary von Neumann \ and \ let
$A$ be a $\ast$-subalgebra in $\ LS(M)$. \ If \  $L$ is a Lie
derivation on  \ $A$ \ then \ $L(Z(A)) \subset  Z(A)$.
\end{lemma}
\begin{proof}
 If \ $z\in Z(A), \ a \in A,$ then \
$[z,a]=0=[z,L(a)]$, \ and consequently \
$[L(z),a]=L([z,a])-[z,L(a)]=0$, i.e. \ $L(z)a=aL(z)$, \ which
implies  that \ $L(z)\in Z(A)$.
\end{proof}

In the von Neumann algebra $M$ we consider the central projection
$$
z_{0}=sup\{z \in \mathcal{P}(Z(M)):zM \subset Z(M)\}.
$$
It is clear, that \ $M_{0}:=z_{0}M=z_{0}Z(M)$ \ is a commutative von
Neumann algebra, and von Neumann algebra
$M_{1}=(\mathbf{1}-z_{0})M$ \ has no direct commutative summands,
moreover \ $M=M_{1}\bigoplus M_{0}$.

Let \ $A$ \ be a solid $\ast$-subalgebra in
  \ $LS(M)$, containing $M$, \
$z_{1}=\mathbf{1}-z_{0}, \ A_{1}=z_{1}A$ \ and \ let \ $L$ \ be a Lie
derivation on $A$.

\begin{lemma}
\label{lemma10} $(i)$. \ $L_{1}(x):=z_{1}L(x), \ x \in A_{1}$ \ is
a Lie derivation on the  solid $\ast$-subalgebra \ $A_{1}$ \ in \
$LS(M_{1})$;

$(ii)$. \ The linear mappings  \ $F_{1}(x):=z_{0}L(z_{1}x), \
F_{2}(x)=z_{1}L(z_{0}x)$ \ and \ $F_{3}(x)=z_{0}L(z_{0}x)$ \ are
$Z(A)$-valued traces on  \ $A$;

$(iii)$. \ If $E_{1}$ is a $Z(A_{1})$-valued traces on  \ $A_{1}$,
\ then \ $E(x):=E_{1}(z_{1}x), \ x\in A$ \ is a \ $Z(A)$-valued
traces on  \ $A$;

$(iv)$. \ If \ $D_{1}$ is an associative  derivation on \
$A_{1}$, \ then \ $D(x):=D_{1}(z_{1}x), \  x\in A$ \ is an associative
derivation on $A$.
\end{lemma}

\begin{proof}
$(i)$. If $x,y\in A_{1}$, then
$$
L_{1}([x,y])=z_{1}L([x,y])=z_{1}([L(x),z_{1}y]+[z_{1}x,L(y)])=[L_{1}(x),y]+[x,L_{1}(y)].
$$

$(ii)$. It is clear, that \
$$
F_{1}(x)\in z_{0}A \subset LS(M_{0})\bigcap A
\subset Z(A).
$$
If \ $x,y\in A$, \ then \ $z_{1}[x,y]=[z_{1}x,z_{1}y]$ \ and \
from equality \ $z_{0}z_{1}=0$ \ it follows that
$$F_{1}([x,y])=
z_{0}(L([z_{1}x,z_{1}y])=z_{0}([L(z_{1}x), \
z_{1}y]+[z_{1}x,L(z_{1}y)])=0,
$$
i. e. \ $F_{1}(xy)=F_{1}(yx)$.

Since \ $z_{0}x\in Z(A)$, \ by Lemma \ref{lemma9}, we have that
$F_{2}(x)=z_{1}L(z_{0}x)\in Z(A)$. \ Moreover, from \
$z_{0}x,z_{0}y\in Z(A)$ \ it follows that $[z_{0}x,z_{0}y]=0$ \
and
$$
F_{2}([x,y])=z_{1}L(z_{0}[x,y])=z_{1}L([z_{0}x,z_{0}y]=0,
$$
i. e. \ $F_{2}(xy)=F_{2}(yx)$. \ Similarly we can show that  \
$F_{3}(xy)=F_{3}(yx)$.

$(iii)$. For any  \ $x,y\in A$ \ we have  \ $E(x)\in
Z(A_{1})\subset Z(A)$ \ and \
$$E(xy)=E_{1}(z_{1}xy)=E_{1}((z_{1}x)(z_{1}y))=E_{1}((z_{1}y)(z_{1}x))=E_{1}(z_{1}yx
)=E(yx).$$

$(iv)$. Since \ $z_{1}D_{1}(a)=D_{1}(a)$ \ for any \ $a\in A_{1}$ \  then
$$D(xy)=D_{1}(z_{1}xy)=D_{1}((z_{1}x)(z_{1}y))=D_{1}(z_{1}x)(z_{1}y)+z_{1}xD_{1}(z_{1}y)=
$$$$D_{1}(z_{1}x)y+xD_{1}(z_{1}y)=D(x)y+xD(y)$$
for all  \ $x,y\in A$.
\end{proof}

Now we will prove Theorem \ref{thm1}. \ If \ $x\in A$, \ then
using  notations of Lemma 10, we obtain
$$L(x)=z_{1}L(z_{1}x)+z_{1}L(z_{0}x)+z_{0}L(z_{1}x)+z_{0}L(z_{0}x)=$$ $$L_{1}(z_{1}x)+F_{2}(x)+F_{1}(x)+F_{3}(x).$$

By Theorem  \ref{thm2} and Lemma \ref{lemma10} $(i)$ \ there exist
an associative  derivation\ $D_{1}$ \ on \ $A_{1}$ \ and a \ $Z(A_{1})$-valued
trace \ $E_{1}$ \ on \ $A_{1}$\ such that\
$L_{1}(z_{1}x)=D_{1}(z_{1}x)+E_{1}(z_{1}x)$ \ for all \ $x\in A$.

By Lemma \ref{lemma10} $(iii)$, $(iv)$, \ the mapping  \
$D(x)=D_{1}(z_{1}x)$ \ is an associative derivation on \ $A$, \ and a mapping  \
$E(x)=E_{1}(z_{1}x)$ \ is a $Z(A)$-valued trace on \ $A$. \ Thus,
$$L(x)=D(x)+E(x)+F_{1}(x)+F_{2}(x)+F_{3}(x) \ \ \ \ \ \ \ \ \ (3.11)$$
for all \ $x\in A$. \ Since  \ $E+F_{1}+F_{2}+F_{3}$ \ is a \
$Z(A)$-valued trace on  \ $A$ \ (see Lemma \ref{lemma10} $(ii)$),
 the equality (3.11) completes the proof of Theorem 3.1.
\hfill $\Box$

\textbf{4. Lie derivations on the  $EW^{*}$ algebras.}

In this section we give applications of Theorem \ref{thm1} to the
description of Lie derivations on $EW^{*}$-algebras. The class of
$EW^{*}$-algebras (extended $W^{*}$-algebras) was introduced in
\cite{Dix} for the purpose of description of  $\ast$-algebras of
unbounded closed operators, which are "similar" \ to
$W^{*}$-algebras by their algebraic and order properties.

Let $\mathcal{A}$ be a set of closed, densely defined operators on
the Hilbert space $H$ which is a unital \ $\ast$-algebra with the
identity \ $\mathbf{1}$ \ equipped with the strong sum, strong
product, the scalar multiplication and the usual adjoint of
operators.

A $\ast$-algebra \ $\mathcal{A}$ \  is said to be \ $EW^*$-algebra
 if the following conditions hold:

$(i)$. $(\mathbf{1}+x^*x)^{-1}\in\mathcal{A}$ \ for every \
$x\in\mathcal{A}$;

$(ii)$. A $\ast$-subalgebra \ $\mathcal{A}_b$ \ of bounded
operators in
 \ $\mathcal{A}$ \ is a \ von Neumann subalgebra in $B(H)$.

In this case, a $\ast$-algebra \ $\mathcal{A}$ \ is said to be
 \ $EW^*$-algebra over von Neumann algebra (or over \ $W^*$-algebra) \ $\mathcal{A}_b$.

The meaningful connection between \ $EW^{*}$-algebras \ $A$ \ and
solid subalgebras of \ $LS(A_{b})$ \ is given in \cite{CZ}.
 It is clear that every solid
$\ast$-subalgebras $A$ in $LS(M)$ with $M\subset A$ is an
$EW^{*}$-algebra and $A_{b}=M.$ The converse implication is given
in \cite{CZ}, where it is established that every $EW^{*}$-algebra
$A$ with the bounded part $A_{b}=M$ is a solid $\ast$-subalgebra
in the $\ast$-algebra $LS(M)$, \ that is $LS(M)$ is the greatest
$EW^{*}$-algebra of $EW^{*}$-algebras with the bounded part
coinciding with $M$.

Therefore,  the
study of derivations in
 � \ $EW^{*}$-algebras is reduced to the study of derivations on solid  $\ast$-subalgebras $A \subset LS(M)$ \ with \ $M \subset A$.

In the case when the bounded part $A_{b}$ of an $EW^{*}$-algebra
$A$ is a property infinite $W^{*}$-algebra we have that any
derivation $\delta:A\rightarrow LS(A_{b})$ is inner \cite[Theorem 5]{BCS2}.

 Now let $M$ be an arbitrary  von Neumann algebra, and $A$ be an  $EW^{*}$-algebra, where \
$A_{b}=M$. \ In this case, as it is mentioned above,  \ $A$ \ is
a solid  $\ast$-subalgebra in  \ $LS(M)$ \ and \ $M\subset A$.  By Theorem 1, any Lie derivation  \ $L$, \ defined on  \ $EW^{*}-$ algebra \ $A$, \ has a standard form  \
$L=D+E$, \ where \ $D:A\rightarrow A$ \ is an associative
derivation, � \ $E:A\rightarrow Z(A)$ \ is a center-valued trace
on \ $A$. \ 

If a von Neumann algebra \ $A$ is properly infinite,
then as it is mentioned above, the derivation \ $D$ \ is inner, i.e. \ $D(x)=[a,x]=D_{a}(x)$ \ for all \ $x\in A$ \ and some a fixed
\ $a\in A$. \ Thus we have the following
\begin{thm}
\label{thm3}
If the bounded part \ $A_{b}$ \ of an \ $EW^{*}$-
algebra \ $A$ \ is a properly infinite \ $W^{*}$-algebra, then every
Lie derivation \ $L$ \ on \ $A$ is equal to \ $D_{a}+E$, \ where \ $a \in A$ \
and \ $E:A \rightarrow Z(A)$ \ is a central-valued trace.
\end{thm}

Let \ $Z$ \ be a commutative von Neumann algebra, let \ $H_{n}$ \ be an
$n$-dimensional complex Hilbert space and let \ 
$M_{n}=B(H_{n})\bigotimes Z$ \ be a homogeneous von Neumann
algebra of type $I_{n}$. The \ von Neumann algebra \ $M_{n}$ \ is
$\ast$-isomorphic to an \ $\ast$-algebra \ $Mat(n,Z)$ \ off all
$n\times n$-matrix \ $(a_{ij})_{i,j=1}^{n}$ with entries  \
$a_{ij} \in Z$. \  Since \ $M_{n}$  is a finite von Neumann
algebra, then  \ $LS(M_n)=S(M_n)$, \ wherein,
 $\ast$-algebra \ $S(M_n)$ \ is identified with the  \ $\ast$-algebra \
$Mat(n,S(Z))$ \ of all \ $n\times n$- matrix with entries from \
$S(M)$ \ (see [1]) .  \ If \ $e_{ij}$ \
$\left(i,j=1,2,...,n\right)$ \ is the matrix unit of \
$Mat(n,S(Z))$, \ then every element  \ $x \in Mat(n,S(Z))$ \ has
the form
 $$
x=\sum_{i,j=1}^{n}\lambda_{ij}e_{ij}, \ \lambda_{ij}\in
S(Z), \ i,j=1,2,..,n.
 $$
 For any derivation  \ $\delta:S(Z)\to S(Z)$ \ the linear operator
 $$
D_{\delta}\bigg(\sum_{i,j=1}^{n}\lambda_{ij}e_{ij}\bigg)=\sum_{i,j=1}^{n}\delta(\lambda_{ij})e_{ij}
\eqno(4.1)
 $$
is a derivation on  \  $S(M_n)=Mat(n,S(Z))$ \  and the restriction of \ $D_\delta$ \ on the center  \ $ Z(S(M_{n}))=S(Z)$ \ coincides with
 $\delta$. \ In the case when a commutative von Neumann algebra \ $Z$ \ has no atoms
    there exists an  uncountable set of mutually different derivations \
$\delta: S(Z) \rightarrow S(Z)$ \ (see \cite{Ber}). Consequently, in this
case, there exists an uncountable family of mutually different
derivations of the form \ $D_{\delta}$
 on the algebra \ $ Mat(n,S(Z))=S(M)$.\

Now let $M$ be an arbitrary finite von Neumann algebra of type I
with the center $Z$.  \ There exists a family \ $\{z_n\}_{n \in F},$ \ $F \subseteq \mathbb N,$ of mutually orthogonal central
projections from \ $M$ \ with \ $\sup\limits_{n\in
F}z_n=\mathbf{1}$ \ such that the algebra \ $M$ is $\ast$
-isomorphic to the direct sum of the von Neumann algebras  \ $z_n
M$ \ of type  $I_n, \ n \in F$, \ i. e.
$$
M \cong \sum_{n\in F} \bigoplus z_nM,
$$
and \
$$
z_{n}M=B(H_{n})\bigotimes
L^{\infty}(\Omega_{n},\Sigma_{n},\mu_{n}) \cong
Mat(k_{n},L^{\infty}(\Omega_{n},\Sigma_{n},\mu_{n}))=
$$
$$
= \{(a_{ij})_{i,j=1}^{n}:
a_{ij}\in L^{\infty}(\Omega_{n},\Sigma_{n},\mu_{n})\}
$$
where  \ $k_{n}=dim(H_{n})<\infty, \ (\Omega_{n},\Sigma_{n},\mu_{n})$
\ is a Maharam measure space, \ $n \in \mathbb N$. According to
Proposition 1.1  from \cite{AAK}, we have that
$$
LS(M)\cong \prod\limits_{n\in F}LS(z_nM),
$$
and \
$$
LS(z_{n}M) \cong Mat(k_{n},L^{0}(\Omega_{n},\Sigma_{n},\mu_{n}))=
$$
$$
= \{(\lambda_{ij})_{i,j=1}^{k_{n}}:
\lambda_{ij}\in
L^{0}(\Omega_{n},\Sigma_{n},\mu_{n})\} \cong B(H_{n})\bigotimes
L^{0}(\Omega_{n},\Sigma_{n},\mu_{n}),
$$
where \ $L^{0}(\Omega_{n},\Sigma_{n},\mu_{n})$ \ is \ an $\ast$-algebra of
all complex measurable functions on \
$(\Omega_{n},\Sigma_{n},\mu_{n})$ \ (equal almost everywhere
functions are identified).

Suppose that $D$ is a derivation on $LS(M)$ and $\delta$ is the
restriction of  $D$ onto the center $S(Z)$. The restriction of the
derivation $\delta$ onto $z_{n}S(Z)$ defines  the derivation
$\delta_n$ on $z_nS(Z)$ for each  $n\in F$.

Let $D_{\delta_{n}}$ be a  derivation on the matrix algebra
$M_{k_{n}}(z_nZ(LS(M)))\cong LS(z_n M)$, defined by the formula
(4.1). Set
$$
D_{\delta}(\{x_n\}_{n\in F})=\{D_{\delta_{n}}(x_n)\}, \ \
\{x_n\}_{n\in F}\in \prod\limits_{n\in F} S(z_nM)\cong LS(M).
\eqno(4.2)
$$

It is clear that  $D_{\delta}$ is a derivation on the algebra
 $LS(M).$

If $M$ is an arbitrary von Neumann algebra of type $I,$
 then there exists a  central projections $z_0 \in M$, such
 that
 \ $z_0 M$  is a finite  von Neumann algebra and
 \ $z_0^{\bot}M$   is a properly infinite algebra.

Consider  derivation $D$ on $LS(M)$ and by $\delta$ denote the
restriction of  $D$ onto the center $Z(LS(M)).$ By Theorem 2.7
\cite{AAK}, the derivation $z_{0}D$ is an inner derivation on
$z_{0}LS(M)=LS(z_{0}M),$ moreover, \ $z_{0}^{\bot}\delta=0$, \ i.
e. \ $\delta=z_0\delta$. \ Let \ $D_{\delta}$ \ be a derivation on
  $z_0 LS(M)$, \ of the form (4.2). Consider the expanding \
$D_{\delta}$ \ onto  \ $LS(M)=z_0 LS(M)\bigoplus z_0^{\bot}LS(M)$,
\ defined as
$$
D_{\delta}(x_1+x_2):=D_{\delta}(x_1), \ \ x_1 \in z_0 LS(M), \ \
x_2 \in z_0^{\bot}LS(M).  \eqno (4.3)
$$

By \cite[Theorems 2.8, 3.6]{AAK}, \  any derivation \ $D$ \ on the
algebra \ $LS(M)$ \ (respectively on the algebra \ $S(M)$) \ can
be uniquely represented as a sum \ $D=D_{a}+D_{\delta}$,  where \
$a$ is a fixed element of \ $LS(M)$ \  (respectively of \ $S(M)$).
 \  This fact and  Theorem 1 imply the following result
\begin{cor}
\label{cor2} Let $M$ be a type $I$ von Neumann algebra. \ Then any
Lie derivationon on  the algebra \ $LS(M)$ \ (respectively on the
algebra \ $S(M)$) has the form
$$
L=D_a+D_\delta+E,
$$
where \ $D_a$  is an inner derivation, \ $D_{\delta}$ is a
 derivation given by (4.3), generated by the
derivation \ $\delta$ \ in the center of \ $LS(M)$ \ (respectively
of \ $S(M)$) and \ $E$  is a center-valued trace on \ $LS(M)$ \
(respectively on \ $S(M)$).
\end{cor}

\end{document}